\documentclass[a4paper,UKenglish,cleveref, autoref, thm-restate]{lipics-v2021}

\hideLIPIcs  

\usepackage{standalone}
\usepackage{tikz}
\usetikzlibrary{arrows,decorations.pathmorphing,backgrounds,positioning,fit,petri,knots,automata,arrows.meta,positioning}
\usepackage[ruled]{algorithm2e}
\usepackage{float}

\newcommand{\DP}{\textnormal{DP}}

\newcommand{\SDP}{\textnormal{SDP}}

\newcommand{\RDP}{\textnormal{RDP}}

\newcommand{\SEP}{\textnormal{SEP}}
\newcommand{\REP}{\textnormal{REP}}

\newcommand{\SAT}{\textnormal{SAT}}

\newcommand{\N}{\mathbb{N}}
\newcommand{\Z}{\mathbb{Z}}

\newcommand{\Col}{\mathcal{C}}

\newcommand{\F}{\mathbb{F}}
\newcommand{\Fo}{\mathcal{F}}

\definecolor{blue}{rgb}{0.29,0.56,0.89}
\definecolor{red}{rgb}{0.82,0.01,0.11}
\definecolor{orange}{rgb}{0.96,0.65,0.14}

\newcommand{\tw}{%
	\tikz[baseline]{
		\draw[fill=white] (-0.3ex,-0.3ex) -- (1.7ex,-0.3ex) -- (1.7ex,1.7ex) -- (-0.3ex,1.7ex) -- cycle;
}}
\newcommand{\tbl}{%
	\tikz[baseline]{
		\draw[fill=blue] (-0.3ex,-0.3ex) -- (1.7ex,-0.3ex) -- (1.7ex,1.7ex) -- (-0.3ex,1.7ex) -- cycle;
}}

\newcommand{\tr}{%
	\tikz[baseline]{
		\draw[fill=red] (-0.3ex,-0.3ex) -- (1.7ex,-0.3ex) -- (1.7ex,1.7ex) -- (-0.3ex,1.7ex) -- cycle;
}}

\newcommand{\torg}{%
	\tikz[baseline]{
		\draw[fill=orange] (-0.3ex,-0.3ex) -- (1.7ex,-0.3ex) -- (1.7ex,1.7ex) -- (-0.3ex,1.7ex) -- cycle;
}}

\title{Contributions to the Domino Problem: Seeding, Recurrence and Satisfiability}

\author{Nicol\'as {Bitar}}{Université Paris-Saclay, CNRS, LISN, 91190 Gif-sur-Yvette, France \and \url{https://www.lri.fr/~bittar/} }{nicolas.bitar@lisn.fr}{https://orcid.org/0000-0002-3460-9442}{}

\authorrunning{N. Bitar} 

\Copyright{Nicol\'as Bitar} 

\ccsdesc[500]{Mathematics of computing~Combinatoric problems} 

\keywords{Tilings, Domino problem, SAT, Computability, Finitely generated groups} 

\category{} 

\relatedversion{} 

\acknowledgements{I want to thank Nathalie Aubrun for her helpful comments, reviews and discussions.}

\nolinenumbers 

\EventEditors{John Q. Open and Joan R. Access}
\EventNoEds{2}
\EventLongTitle{42nd Conference on Very Important Topics (CVIT 2016)}
\EventShortTitle{CVIT 2016}
\EventAcronym{CVIT}
\EventYear{2016}
\EventDate{December 24--27, 2016}
\EventLocation{Little Whinging, United Kingdom}
\EventLogo{}
\SeriesVolume{42}
\ArticleNo{1}

\begin{document}

\maketitle

\begin{abstract}
We study the seeded domino problem, the recurring domino problem and the $k$-SAT problem on finitely generated groups. These problems are generalization of their original versions on $\mathbb{Z}^2$ that were shown to be undecidable using the domino problem. We show that the seeded and recurring domino problems on a group are invariant under changes in the generating set, are many-one reduced from the respective problems on subgroups, and are positive equivalent to the problems on finite index subgroups. This leads to showing that the recurring domino problem is decidable for free groups. Coupled with the invariance properties, we conjecture that the only groups in which the seeded and recurring domino problems are decidable are virtually free groups. In the case of the $k$-SAT problem, we introduce a new generalization that is compatible with decision problems on finitely generated groups. We show that the subgroup membership problem many-one reduces to the $2$-SAT problem, that in certain cases the $k$-SAT problem many one reduces to the domino problem, and finally that the domino problem reduces to $3$-SAT for the class of scalable groups.
\end{abstract}

\section{Introduction}
\label{sec:intro}

The domino problem, as originally formulated, is the decision procedure of determining if a given finite set of unit squares with colored edges, known as Wang tiles, can tile the infinite plane while respecting an adjacency condition: two squares can be placed next to each other if there shared edge has the same color. This problem was introduced by Wang to study the decidability of the $\forall\exists\forall$ fragment of first order logic \cite{wang1961proving}, who conjectured that the domino problem was decidable. It turns out that this is not the case; a now classic result by Berger states that this problem is undecidable \cite{berger1966undecidability}. Since then, many proofs of the undecidability of the problem have been found (see \cite{jeandel2020domino}). 

Perhaps one of the fundamental features of the domino problem is its usefulness in proving the undecidability of many decision problems, ranging from problems in symbolic dynamics such as the infinite snake problem \cite{adleman2009snake} and the injectivity and surjectivity of two-dimensional cellular automata \cite{kari1990reversibility,kari1994automata}, to problems from other areas such as the $k$-SAT problem on $\Z^2$ \cite{freedman1999sat}, the spectral gap problem of quantum many-body systems \cite{cubitt2022undecidability} and translation monotilings~\cite{greenfeld2023undecidability}. In fact, the Wang tiling model can be seen as a natural model to encode computation and prove complexity lower bounds \cite{vanEmde1997convinience}.

In recent years, the domino problem has found new life in the context of symbolic dynamics over finitely generated groups \cite{aubrun2018domino}. The aim has been to establish which algebraic conditions make the problem of deciding whether the group is tileable subject to a finite number of local constraints --- i.e., the domino problem on the group --- undecidable. This project has culminated in a conjecture stating that the class of groups with decidable domino problem is the class of virtually free groups. This same turn towards generalizing to groups has been made for different problems, such as aperiodic tilings \cite{rieck2022strongly} and domino snake problems \cite{aubrun2023snake}. This article follows the same path for three different decision problems: the seeded domino problem, the recurrent domino problem, and the $k$-SAT problem. We introduce generalizations of these problems for finitely generated groups to understand how the algebraic and computational properties of the underlying group influence the decidability and to obtain new tools to show the undecidability of problems on finitely generated groups.

\paragraph*{Variants of Tiling Problems}

Variations on the domino problem have been present since its conception. This is the case of the seeded domino problem, whose undecidability was established even before the domino problem's \cite{kahr1962entscheidungsproblem,buchi1962turing}. In the years since many more variations have been introduced: the periodic domino problem \cite{gurevich1972domino,jeandel2010periodic}, domino snake problems \cite{myers1979decidability}, the recurrent domino problem \cite{harel1985recurring}, the aperiodic domino problem \cite{callard2022aperiodic, grandjean2018aperiodic}, and even a variant where the underlying structure is fixed to be a geometric tiling of the plane \cite{hellouin2023rhombus}.

In this article, we generalize the seeded and recurrent problems to the context of finitely generated groups. The seeded version asks, given an alphabet, a finite set of local rules and a target letter from the alphabet, if there exists a coloring of the group subject to the local rules where the target letter appears. The recurrent version has the same input but asks if such a configuration exists where the target letter appears infinitely often. We will begin by formally introducing these problems and establishing connections to the domino problem: the problem many-one reduces to both variants (Section 2). Next, we show that the decidability of both problems is independent of the chosen generating set for the group, that the problems many-one reduce from subgroups and are in fact equivalent in the case of finite index subgroups (Section 3). Furthermore, we show that the recurrent problem is decidable on free groups, which paired with the domino conjecture and inheritance properties, allows us to state the following extension of the domino conjecture

\begin{conjecture}
Let $G$ be a finite generated group. The following are equivalent,
		\begin{itemize}
			\item $G$ is virtually free,
			\item the domino problem on $G$ is decidable,
			\item the seeded domino problem on $G$ is decidable,
			\item the recurrent domino problem on $G$ is decidable.
		\end{itemize}
\end{conjecture}

\paragraph*{$k$-SAT and the Limit of Polynomial Time Problems}

The $k$-SAT problem on groups was introduced by Freedman in \cite{freedman1999sat}. The idea of the generalization was to extend the difference between $2$-SAT and $3$-SAT, which are in $\textbf{P}$ and $\textbf{NP}$ respectively, to an infinite context making the former problem decidable and the latter undecidable. This is inserted into the broader program outlined in \cite{freedman1998limit} that searches to separate the complexity classes $\textbf{P}$ and $\textbf{NP}$ by limit processes, the idea being that limiting behaviors of polynomial time problems should be decidable. 

In this article, we slightly alter the generalization proposed by Freedman to make the decision problem compatible with finitely generated groups (Section 4). Similar generalizations have been made for other classic decision problems, such as Post's correspondence problem \cite{myasnikov2014post,ciobanu2021variations,ciobanu2022post}. We show that the subgroup membership problem of the group many-one reduces to the complement of the $2$-SAT problem and that in the class of groups where the former is decidable, the $k$-SAT problem many-one reduces to the domino problem for all $k>1$. In conjunction with the work of Piantadosi \cite{piantodosi2008free} and the domino problem's inheritance properties, this result implies that the $k$-SAT problem is decidable for virtually free groups. We introduce the class of scalable groups to find an equivalence between both decision problems. A finitely generated group is scalable if it contains a proper finite index subgroup that is isomorphic to the group. We show that for this class of groups, the domino problem many-one reduces to the $3$-SAT problem. The proof of this result is inspired by techniques from~\cite{freedman1999sat} that are adapted to better suit our definition of the decision problem.\\

We begin the article by introducing some preliminary notions from computability and symbolic dynamics in Section 1. Section 2 introduces the domino problem, the seeded and recurrent variants, and some basic properties. In Section 3, we establish the invariance of the decidability of the seeded and recurrent problems under changes in the generating set of the group, in addition to some inheritance properties, most notably subgroups. We also prove that the recurrent domino problem is decidable for free groups and show that the seeded and recurrent problems are subject to the same conjecture as the normal one. Section 4 is devoted to the study of the $k$-SAT problem on finitely generated groups. We formally introduce the decision problem as well as its connection to the subgroup membership problem and a reduction to the domino problem. We finally present the class of scalable groups and show that for them, the domino problem many-one reduces to the $3$-SAT problem. 

\section{Preliminaries}

Given a finite alphabet $A$, we denote the set of words of length $n$ by $A^n$, the set of words of length less or equal to $n$ by $A^{\leq n}$, and the set of all finite words over $A$ by $A^*$, including the empty word $\varepsilon$. The length of a word $w$ is denoted by $|w|$. We denote the free group on $n$ generators by $\F_n$. Throughout the article, $G$ will be an infinite group.

\subsection{Computability and Group Theory}

We quickly recall some notions from computability theory and combinatorial group theory that will be needed in the article. See \cite{Rogers} for a reference on computability and reductions, and see \cite{LyndonSchupp} for a reference on combinatorial group theory.

\begin{definition}
	Let $L\subseteq A^*$ and $L'\subseteq B^*$ be two languages. We say,
	\begin{itemize}
		\item $L$ many-one reduces to $L'$, denoted $L\leq_m L'$, if there exists a computable function $f:A^*\to B^*$ such that $w\in L$ if and only if $f(w)\in L'$ for every $w$.
		
		\item $L$ positive-reduces to $L'$, denoted $L\leq_p L'$ if for any $w$ one can compute finitely many finite sets $F_1(w), ..., F_n(w)$ such that $w\in L$ if and only if there exists $i\in\{1, ..., n\}$ such that $F_i(w)\subseteq L'$.
	\end{itemize}
	For both notions of reducibility, the induced notion of equivalence will be denoted by $L\equiv_* L'$, meaning $L\leq_* L'$ and $L'\leq_* L$.
\end{definition}
Notice that many-one reducibility implies positive-reducibility. The complement of a decision problem $D$, denoted $\textnormal{co}D$, is the set of all "no" instances of $D$.\\

Given $G$ a finitely generated group (f.g.) and $S$ a finite generating set, elements in the group are represented as words over the alphabet $S\cup S^{-1}$ through the evaluation function $w\mapsto \overline{w}$. Two words $w$ and $v$ represent the same element when $\overline{w} = \overline{v}$, and we denote it by $w =_G v$. We say a word is \emph{reduced} if it contains no factor of the form $ss^{-1}$ or $s^{-1}s$ with $s\in S$. The length of an element $g\in G$ with respect to $S$, denoted $|g|_S$, is the length of a shortest word $w\in (S\cup S^{-1})^*$ such that $g = \overline{w}$. A group is \emph{virtually free} if it contains a finite index subgroup isomorphic to a free group.

\subsection{Subshifts of Finite Type}

Let $A$ be a finite alphabet and $G$ a finitely generated group. The \emph{full-shift} on $A$ is the set of configurations $A^G = \{x\colon G\to A\}$. This space is acted upon by $G$ in the form of left translations: given $g\in G$ and $x\in A^G$,
$$g\cdot x(h) = x(g^{-1}h).$$

Let $F$ be a finite subset of $G$. We call $p\in A^{F}$ a f\emph{pattern} of support $F$. We say a pattern $p$ appears in a configuration $x\in A^G$ if there exists $g\in G$ such that $p(h) = x(gh)$ for all $h\in F$. The \emph{cylinder} defined by a pattern $p\in A^F$ at $g\in G$ is given by 
$$[p]_g = \{x\in A^G \colon\ \forall h\in F,\ x(gh) = p(h)\}.$$

Given a set of patterns $\Fo$, we define the $G$-\emph{subshift} $X_{\Fo}$ as the set of configurations where no pattern from $\Fo$ appears. That is,
$$X_{\mathcal{F}} = \{x\in A^G \colon\ \forall p\in\mathcal{F},\ p \text{ does not appear in } x\} = A^G\setminus\bigcap_{g\in G, p\in\Fo}[p]_g.$$
If $\mathcal{F}$ is finite, we say $X_{\mathcal{F}}$ is a $G$-\emph{subshift of finite type} ($G$-SFT). We will simply write SFT when the group is clear from context.\\

Let $S$ be a finite generating set for $G$. We say a pattern $p$ is \emph{nearest neighbor} if its support is given by $\{1_G, s\}$ with $s\in S$. We will denote nearest neighbor patterns through tuples $(a,b,s)$ representing $p(1_G) = a$ and $p(s) = b$. A subshift defined by a set of nearest neighbor forbidden patterns is known as a nearest neighbor subshift. These subshifts are necessarily SFTs. Given a set of nearest neighbor patterns $\Fo$, we define its corresponding tileset graph, $\Gamma_{\Fo}$, by the set of vertices $A$, and edges given by $(a,b,s)\not\in\Fo$, where $a$ is its initial vertex, $b$ its final vertex and $s$ its label (see Figure \ref{fig:tileset} for an example).
\begin{figure}[h]
	\centering
	\begin{tikzpicture}[node distance = 3cm, on grid, auto]
	\node (q0) [state] at (1.5,1) {$0$};
	\node (q1) [state] at (3,3) {$1$};
	\node (q2) [state] at (4.5,1) {$2$};
	
	\path [-stealth, thick]
	(q0) edge [bend left] node {$s$}	(q1)
	(q1) edge [bend left] node {$s$}   (q2)
	(q2) edge [bend left] node {$s$}   (q0)
	
	(q0) edge node {$t$}   (q2)
	(q1) edge node {$t$} (q0)
	
	(q0) edge [loop left]  node {$s$}()
	(q1) edge [loop above]  node {$t$}()
	(q2) edge [loop right]  node {$t$}();
	
	
\end{tikzpicture}
	\caption{An example of a tileset graph $\Gamma_{\Fo}$ for the alphabet $\{0,1,2\}$ and a group with generators $s$ and $t$. Edges present in the graph are exactly those that are not in $\Fo$, for example $(2,1,t)\in\Fo$.}
	\label{fig:tileset}
\end{figure}
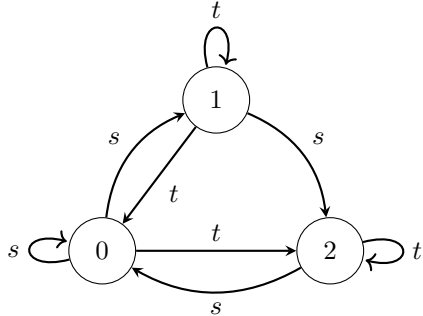	

These graphs will help us in Section \ref{sec:free} when working with free groups. It is a well known fact \cite{LindMarcus} that nearest neighbor SFTs over $\Z$ are characterized as the set of bi-infinite walks on their corresponding tileset graph.

\section{The Domino Problem and its Variants}

We begin with a formal definition of the domino problem that generalizes the original formulation with Wang tiles.

\begin{definition}Let $G$ be a finitely generated group and $S$ a finite generating set. The \emph{domino problem} on $G$ with respect to $S$ is the decision problem that, given an alphabet $A$ and a finite set of nearest neighbor forbidden patterns $\Fo$, determines if the corresponding subshift $X_{\Fo}$ is empty. We denote this problem by $\DP(G,S)$.
\end{definition}

It has been shown that the decidability of the problem is invariant under changes in the generating set, that is, if $S_1$ and $S_2$ are two finite generating sets for $G$, then $\DP(G,S_1)\equiv_m\DP(G,S_2)$. We can therefore talk about \textit{the} domino problem on $G$, denoted $\DP(G)$. The problem also satisfies many inheritance properties: both the domino problem of a finitely generated subgroup and the domino problem of a quotient by a finitely generated kernel many-one reduce to the domino problem of the group. Furthermore, the word problem of the group many-one reduces to the complement of the domino problem. Proofs of these facts can be found in~\cite{aubrun2018domino}.

A particularly important property enjoyed by the problem, is that it can be expressed in the monadic second order (MSO) logic of the group's Cayley graph \cite{ballier2008tilings}. Coupled with the fact that virtually free groups have decidable MSO logic \cite{muller1985ends,kuske2005logical}, this implies that the domino problem is decidable for virtually free groups. It is possible to go even further, this characterization of virtually free groups tells us that if a group is not virtually free its Cayley graphs contain arbitrarily large grids as minors \cite{robertson1986minor}. This prompted Ballier and Stein \cite{ballier2018domino} to state the following conjecture.

\begin{conjecture}[The Domino Conjecture]
	\label{conj:domino}
	Let $G$ be a finitely generated group. Then $\DP(G)$ is decidable if and only if $G$ is virtually free.
\end{conjecture}

Since then, many classes of groups have been shown to satisfy the conjecture, such as Baumslag-Solitar groups \cite{aubrun2013baumslag}, polycyclic groups \cite{jeandel2015aperiodic}, hyperbolic groups \cite{bartholdi2023domino}, Artin groups \cite{aubrun2023gbs}, direct products of two infinite groups \cite{jeandel2015translation}, among others.

\subsection{Seeded Domino Problem}

Perhaps the most natural variant of the domino problem is its seeded version. In fact, it was introduced simultaneously to the original problem \cite{wang1961proving} and, as previously mentioned, was shown to be undecidable on $\Z^2$ before the domino problem \cite{buchi1962turing,kahr1962entscheidungsproblem}.

\begin{definition}
	Let $G$ be a finitely generated group and $S$ a finite generating set. The \emph{seeded domino problem} on $G$ with respect to $S$ is the decision problem that, given an alphabet $A$, a finite set of nearest neighbor forbidden patterns $\Fo$ and a letter $a_0\in A$, determines if there exists $x\in X_{\Fo}$ such that $x(1_G) = a_0$. We denote the decision problem by $\SDP(G,S)$.
\end{definition}

As its definition suggests, this problem is computationally harder than the unseeded version: for a set of nearest neighbor forbidden patterns $\Fo$ over the alphabet $A$, we create an instance of the seeded domino problem per letter. 

\begin{lemma}
\label{lem:dp_to_sdp}
	Let $G$ be a finitely generated group and $S$ a finite generating set. Then, $\DP(G,S)\leq_p\SDP(G,S)$.
\end{lemma}

Just as the domino problem, there is a behavioral jump from the one-dimensional case to the two-dimensional case. Using the fact that nearest neighbor $\Z$-SFTs are defined as bi-infinite walks on a finite graph, $\SDP(\Z, \{t\})$ is decidable. This difference in computability prompts the study of this problem on finitely generated groups. In fact, the problem can be shown to be decidable on the entire class of virtually free groups. Just as the domino problem, the seeded version can be expressed in monadic second order logic (see \cite{bartholdi2022mso}), making the problem decidable in this class.

\subsection{Recurring Domino Problem}


The recurring domino problem was originally introduced by Harel as a natural decision problem that is highly undecidable, in order to find other highly undecidable problems \cite{harel1985recurring}. He showed that in $\Z^2$ the problem is not only undecidable, but it is beyond the arithmetical hierarchy: it is $\Sigma_1^1$-complete \cite{Harel_1986}. We expand the problem's definition to finitely generated groups.

\begin{definition}
	Let $G$ be a finitely generated group and $S$ a finite generating set. The \emph{recurrent domino problem} on $G$ with respect to $S$ is the decision problem that, given an alphabet $A$, a finite set of nearest neighbor forbidden patterns $\Fo$ and a letter $a_0\in A$, determines if there exists $x\in X_{\Fo}$ such that the set $\{g\in G \colon\ x(g) = a_0\}$ is infinite. We denote the decision problem by $\RDP(G,S)$.
\end{definition}

As was the case with the seeded variant, this problem is computationally harder that the standard domino problem.

\begin{lemma}
\label{lem:dp_to_rdp}
	Let $G$ be a finitely generated group and $S$ a finite generating set. Then, $\DP(G,S)\leq_p\RDP(G,S)$.
\end{lemma}

\begin{proof}
	Let $\Fo$ be a set of nearest neighbor patterns for $\DP(G,S)$. Notice that this is already part of the input for the recurring version, we simply create an instance for $\RDP(G,S)$ for each of the letters of the alphabet. Because $G$ is infinite, if the subshift defined by $\Fo$ is non-empty, then at least one letter is forced to repeat itself infinitely often.
\end{proof}
Nevertheless, the behavioral jump that occurs between $\Z$ and $\Z^2$ for the original problem is still present.

\begin{proposition}
	$\RDP(\Z,\{t\})$ is decidable. 
\end{proposition}

\begin{proof}
	Let $\Fo$ be a finite set of nearest neighbor forbidden patterns and $a_0\in A$. Recall that we can define a graph $\Gamma_{\Fo}$, that is effectively constructible from $\Fo$, such that configurations on $X_{\Fo}$ correspond exactly with bi-infinite walks on $\Gamma_{\Fo}$. Therefore, to decide our problem we simply have to search for a simple cycle on $\Gamma_{\Fo}$ that is based at $a_0$. If there is such a cycle, $c = a_0 a_1 a_2 ... a_n a_0$ with $a_i\in A$, we define the periodic configuration $x = (a_0a_1 a_2 a_3 ... a_n)^{\infty}\in X_{\Fo}$. If, on the other hand, there exists a configuration $y\in X_{\Fo}$ on which $a_0$ appears infinitely often; take two consecutive occurrences of $a_0$, say $y(k) = y(k') = a_0$ with $k<k'$. Then, because configurations correspond to bi-infinite walks, there is a cycle on $\Gamma_{\Fo}$ given by $c' = a_0 y(k+1) y(k+2) \ ... \ y(k'- 1)a_0$. As searching for simple cycles on a finite graph is computable, our problem is decidable.
\end{proof}

\section{Properties for Seeded and Recurring Variants}

\subsection{General Inheritance Properties}

Let us try and recover some inheritance properties enjoyed by the standard domino problem for the two variants, starting by the invariance under changing generating sets. We use strategies and procedures used to prove the corresponding results for the normal problem, as done in~\cite{aubrun2018domino}. We begin by making use of pattern codings, which are a computationally tractable way of defining forbidden patterns whose support is not $\{1_G, s\}$.

\begin{definition}
	Let $G$ be a f.g. group, $S$ a finite set of generators and $A$ a finite alphabet. A pattern coding $c$ is a finite set of tuples $c = \{(w_i, a_i)\}_{i\in I}$, where $w_i\in(S\cup S^{-1})^*$ and $a_i\in A$.
\end{definition}

Given a set of pattern codings $\Col$, we define its corresponding subshift as:

$$X_{\Col} = A^{G}\setminus \bigcup_{g\in G}\bigcap_{\substack{c\in\Col \\ (w,a)\in c}}[a]_{gw}.$$

\begin{definition}
	The seeded (recurrent) emptiness problem on $G$ with respect to $S$ asks if, given $\Col$ a set of pattern codings and $a_0\in A$, there exists $x\in X_{\Col}$ such that $x(1_G) = a_0$ (resp. $a_0$ appears infinitely often in $x$, that is, $|\{g\in G \colon\ x(g) = a_0\}|$ is infinite).
\end{definition}

Let us denote this problem by $\SEP(G,S)$ (resp. $\REP(G,S)$).

\begin{lemma}
	\label{cambiogen}
	Let $G$ be a f.g. group along with two finite generating sets $S_1$ and $S_2$. Then,
	\begin{itemize}
		\item $\SDP(G, S_1) \equiv_p \SDP(G, S_2)$,
		\item $\RDP(G, S_1) \equiv_p \RDP(G, S_2)$.
	\end{itemize}
\end{lemma}

\begin{proof}
	
	We begin by noticing that we can re-write a pattern coding into any other generating set. This means that $\SEP(G, S_1)\equiv_m\SEP(G, S_2)$. We therefore just need to show that $\SDP(G,S)\equiv_p\SEP(G,S)$.
	
	It is straight forward to re-write nearest neighbor patterns as pattern codings: each pattern $(a,b,s)$ becomes $c = \{(\varepsilon, a), (s, b)\}$. Thus, we have the reduction $\SDP(G,S)\leq_m\SEP(G,S)$. Let us now focus on proving that $\SEP(G,S)$ positive-reduces to $\SDP(G, S)$. Given a set of pattern codings $\Col$ and a letter $a_0$, we can compute both
	$$ N = \max_{c\in\Col}\max_{(w,a)\in c}|w|,$$
	and a new alphabet, $\hat{A}$, consisting of colorings of words of length at most $N$ with no pattern from $\Col$:
	$$\hat{A} = \left\{\phi: S^{\leq N}\to A \ \colon\ \forall c\in\Col, \exists(w,a)\in c,\ \phi(w)\neq a\right\}.$$
	
	In addition, we are able to compute a set of forbidden patterns over $\hat{A}$ denoted by $\Fo'$ such that:
	$$q\in\Fo' \iff q\in \hat{A}^{\{1,s\}}: \ \exists w\in S^{\leq N-1}, \ q_1(sw)\neq q_{s}(w).$$
	Finally, we compute the set of all functions $\phi\in\hat{A}$ such that $\phi(\varepsilon) = a_0$, and denote this set by $\mathcal{A}$. We create $|\mathcal{A}|$ sets of inputs for $\SDP(G,S)$ given by the forbidden patterns $\Fo'$ and a target letter $\phi\in\mathcal{A}$. \\
	
	If there exists a configuration $x\in X_{\Col}$ such that $x(1_G) = a_0$, we define $y\in X_{\Fo'}\subseteq \hat{A}^G$ by $y(g)(w) = x(g\bar{w})$. This way, $y$ contains no pattern from $\Fo'$, as $x$ does not contain a pattern coding from $\Col$, and $y(1_G)(\varepsilon) = x(1_G) = a_0$. Conversely, if there is function $\phi\in\mathcal{A}$ and a configuration $y\in X_{\Fo'}$ such that $y(1_G) = \phi$, we construct $x\in X_{\Col}$ by $x(g) = y(g)(\varepsilon)$. From the definition of $\Fo'$, if we take $g\in G$ with $|g|_S\leq N$ and a word $w\in S^{\leq N}$ such that $\bar{w} = g$, we have that $y(1_G)(w) = y(g)(\varepsilon)$. Therefore, $x$ is well-defined and contains no pattern codings from $\Col$. Furthermore, $x(1_G) = y(1_G)(\varepsilon) = a_0$.
	
	All the previous arguments are analogous for the case of $\RDP(G,S)$ and $\REP(G,S)$.
\end{proof}

This Lemma allows us to talk about \emph{the} seeded domino problem on $G$, $\SDP(G)$, and \emph{the} recurring domino problem on $G$, $\RDP(G)$.

\begin{lemma}
	\label{lem:subgrupo}
	Let $G$ be a f.g. group along with a finitely generated subgroup $H$. Then,
	\begin{itemize}
		\item $\SDP(H) \leq_m \SDP(G)$,
		\item $\RDP(H) \leq_m \RDP(G)$.
	\end{itemize}
\end{lemma}

\begin{proof}
	Let $S_H$ and $S_G$ be finite sets of generators for $H$ and $G$ respectively. We will work with the seeded version, as the recurring one is analogous. Notice that an instance, $(\Fo, a_0)$, of $\SDP(H, S_H)$ is also an instance of $\SDP(G, S_G\cup S_H)$. 
	
	Now, if there exists $x\in X_{\Fo}\subseteq A^G$ with $x(1_G) = a_0$, then the configuration $y = x|_H\in A^H$ contains no patterns from $\Fo$ and verifies $y(1_H) = a_0$. On the other hand, if there exists $y\in X_\Fo\subseteq A^H$; let $L$ be a set of left representatives for $G/H$. We define $x\in A^G$ as $x(lh) = y(h)$ for all $l\in L$ and all $h\in H$. Because the forbidden patterns are supported on $S_H$, we have that $x\in X_\Fo\subseteq A^G$.
\end{proof}

\begin{lemma}
\label{lem:fi}
	Let $G$ be a f.g. group along with a subgroup $H$ such that $[G:H]<\infty$. Then,
	\begin{itemize}
		\item $\SDP(G) \equiv_p \SDP(H)$,
		\item  $\RDP(G) \equiv_p \RDP(H)$.
	\end{itemize}
\end{lemma}

\begin{proof}
	Because finite index subgroups of finitely generated groups are finitely generated, $\SDP(H) \leq_m \SDP(G)$ by Lemma \ref{lem:subgrupo}. We now prove that $\SDP(H)\leq_p \SDP(G)$. Without loss of generality, we may assume $H\trianglelefteq G$: every finite index subgroup $H$ contains a normal finite index subgroup $N$, then if we prove $\SDP(G)$ reduces to $\SDP(N)$, we can conclude it reduces to $\SDP(H)$ by Lemma \ref{lem:subgrupo}. \\

	Let $X\subseteq A^G$ be a subshift, and $R$ a set of right co-set representatives for $G/H$, containing the identity $1_G$. We define what is known as the $R$-higher power shift of $X$ as: 
	$$ X^{[R]} = \{y\in (A^R)^H \ \colon \ \exists x\in X: \ \forall(h,r)\in H\times R, \ y(h)(r) = x(hr) \}.$$
	
	It is clear that $X^{[R]}$ is an $H$-subshift, and we will show that if $X$ is a $G$-SFT, then $X^{[R]}$ is a $H$-SFT. Let $S_H$ be a finite set of generators for $H$. We define the sets $D = S_H \cup (RRR^{-1}\cap H)$ and $T = RDR^{-1}$. Because $1_G\in R$ and $H$ is a normal subgroup, $H = \langle T\rangle$.
	
	We will positive-reduce $\SDP(G, S_H\cup R)$ to $\SDP(H,T)$. Let $(\Fo, a_0)$ be an instance of $\SDP(G, S_H\cup R)$. Let us construct a set $\Fo'$ of forbidden patterns over the alphabet $A^R$, such that $X_{\Fo'} = X_{\Fo}^{[R]}$. We begin by defining the set of $R$-patterns containing $a_0$:
	$$\mathcal{A} = \{p\in A^R \ \colon \ p(1_G) = a_0\}.$$
	
	Now, take $(a,b,s)\in\Fo$. We will add patterns to $\Fo'$ depending on where $s$ belongs.
	\begin{itemize}
		\item If $s\in S_H$, we add for each $r\in R$ a pattern $q$ of support $\{1_H, rsr^{-1}\}$ such that $q(1_H)(r) = a$ and $q(rsr^{-1})(r) = b$.

		\item If $s\in R$, notice that for any $r\in R$, we have $rs = hr'$ where $r'\in R$ and $h\in RRR^{-1}\cap H$, as $R$ is a set of right coset representatives. Therefore, for each $r\in R$, $h$ and $r'$ as befor, we add the pattern $q$ of support $\{1_H, h \}$ such that $q(1_H)(r) = a$ and $q(h)(r') = b$. 
	\end{itemize}
	
	A straightforward computation shows $X_{\Fo'} = X_{\Fo}^{[R]}$. Finally, we create $|\mathcal{A}|$ inputs for $\SDP(H,T)$ given by $\Fo'$ and a letter from $\mathcal{A}$. Suppose there exists $x\in X_{\Fo}$ such that $x(1_G) = a_0$. Define $y\in X_{\Fo}^{[R]}$ as $y(h)(r) = x(hr)$ for all $h\in H$, $r\in R$, which implies $y(1_H) = x|_{R}\in\mathcal{A}$. Conversely, if there exists $y\in X^{[R]}_{\Fo}$ such that $y(1_H)\in \mathcal{A}$, define $x\in X_{\Fo}$ by $x(hr) = y(h)(r)$ for all $h\in H$ and $r\in R$. Thus, $x(1_G) = y(1_H)(1_G) = a_0$.
	
	Because $|R|<+\infty$, the case for $\RDP$ is analogous.
\end{proof}

\subsection{Recurring Domino Problem on Free Groups}
\label{sec:free}

In this section we prove the following result:

\begin{theorem}
\label{thm:free}
	$\RDP(\F_n)$ is decidable.
\end{theorem}
 
Fix $S$ a free generating set for $\F_n$. Let $A$ be an alphabet, $\Fo$ a set of nearest neighbor forbidden patterns and $a_0\in A$ the tile we want an infinity of.  The goal of our algorithm will be to find a particular structure within the tileset graph $\Gamma_{\Fo}$ called a \emph{simple balloon}. We will then show that there is such a structure if and only if there is a configuration in $X_{\Fo}$ where $a_0$ occurs infinitely often. 

\begin{definition}
A \emph{balloon} $B$ is an undirected path in $\Gamma_{\Fo}$, starting and ending at $a_0$, which is specified by a sequence of letters and generators $B = a_0 s_1 a_1 \ ... \ s_{n-1} a_{n-1} s_n a_0$, with $a_i\in A$, $s_i\in S\cup S^{-1}$, such that $s_i = s$ if $(a_{i-1},a_{i},s)$ is an edge in $\Gamma_{\Fo}$, and $s_i = s^{-1}$ if $(a_i, a_{i-1},s)$ is an edge in $\Gamma_{\Fo}$, its label $s_1\ ...\ s_n$ is reduced, and if there exists $k\leq \lceil\frac{n}{2}\rceil-1$ such that $$s_1 \ ... \ s_k = (s_{n-k+1} \ ... \ s_{n})^{-1},$$ then $a_i = a_{n-i}$ for $i\in\{1,...,k\}$. We say the balloon is \emph{simple} if for every $i\in\{1,...,n-1\}$ the pair $a_i s_{i+1}$ never repeats.
\end{definition}

The last condition in the definition of a balloon asks that if the label is not cyclically reduced, $w = uvu^{-1}$ for instance, then the first $|u|$ must be the same as the last $|u|$ tiles in reverse order (see Figure \ref{fig:balloon}).

Given a simple balloon, we want to create a configuration by repeating the letter/generator sequence it defines. Nevertheless, this only covers a portion of the group. To guarantee we will be able to complete a configuration we must ask for each letter to have the ability to be extended to cover the whole group, and thus any portion.
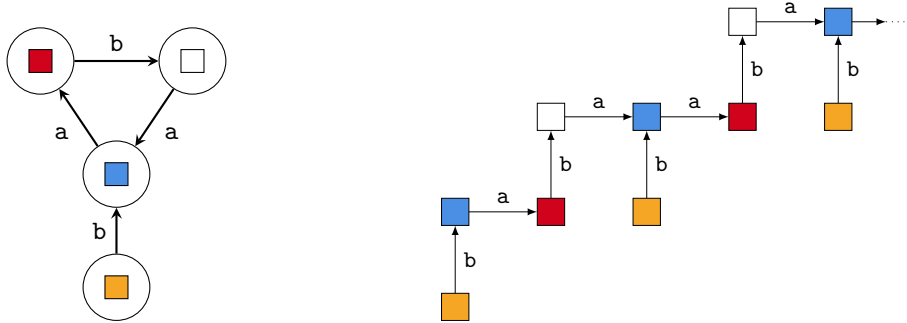
\begin{figure}[h]
	\centering
	\begin{subfigure}{0.4\textwidth}
		\begin{tikzpicture}[node distance = 0.5cm, on grid, auto]
	\node (q0) [state] at (1,0.5) {$\tbl$};
	\node (q1) [state] at (0,2) {$\tr$};
	\node (q2) [state] at (2,2) {$\tw$};
	\node (q3) [state] at (1,-1) {$\torg$};

	\path [-stealth, thick]
	(q0) edge node {$\mathtt{a}$}   (q1)
	(q1) edge node {$\mathtt{b}$}   (q2)
	(q3) edge node {$\mathtt{b}$}   (q0)
	(q2) edge node {$\mathtt{a}$}   (q0);
	
	
\end{tikzpicture}
	\end{subfigure}
	\begin{subfigure}{0.45\textwidth}
		\begin{tikzpicture}

%
%
%

\draw[fill=orange] (1,1) -- (1,1.4) -- (1.4,1.4) -- (1.4,1) -- cycle;
\draw [-latex] (1.2, 1.4) -- (1.2,2.4) node[midway,right] {$\mathtt{b}$};

\draw[fill=blue] (1,2.4) -- (1,2.8) -- (1.4,2.8) -- (1.4,2.4) -- cycle;

\foreach \x in {0,1}{
\draw [-latex] (1.4+2.8*\x, 2.6+1.4*\x) -- (2.4+2.8*\x,2.6+1.4*\x) node[midway,above] {$\mathtt{a}$};
	
\draw[fill=red] (2.4+2.8*\x,2.4+1.4*\x) -- (2.4+2.8*\x,2.8+1.4*\x) -- (2.8+2.8*\x,2.8+1.4*\x) -- (2.8+2.8*\x,2.4+1.4*\x) -- cycle;
\draw [-latex] (2.6+2.8*\x, 2.8+1.4*\x) -- (2.6+2.8*\x,3.8+1.4*\x) node[midway,right] {$\mathtt{b}$};

\draw[fill=white] (2.4+2.8*\x,3.8+1.4*\x) -- (2.4+2.8*\x,4.2+1.4*\x) -- (2.8+2.8*\x,4.2+1.4*\x) -- (2.8+2.8*\x,3.8+1.4*\x) -- cycle;
\draw [-latex] (2.8+2.8*\x, 4+1.4*\x) -- (3.8+2.8*\x,4+1.4*\x) node[midway,above] {$\mathtt{a}$};

\draw[fill=blue] (3.8+2.8*\x,3.8+1.4*\x) -- (3.8+2.8*\x,4.2+1.4*\x) -- (4.2+2.8*\x,4.2+1.4*\x) -- (4.2+2.8*\x,3.8+1.4*\x) -- cycle;
\draw [-latex] (4+2.8*\x, 2.8+1.4*\x) -- (4+2.8*\x,3.8+1.4*\x) node[midway,right] {$\mathtt{b}$};

\draw[fill=orange] (3.8+2.8*\x,2.4+1.4*\x) -- (3.8+2.8*\x,2.8+1.4*\x) -- (4.2+2.8*\x,2.8+1.4*\x) -- (4.2+2.8*\x,2.4+1.4*\x) -- cycle;
}
\draw [-latex] (7,5.4) -- (7.5,5.4);
\draw [dotted] (7.5,5.4) -- (7.8,5.4);

\end{tikzpicture}
	\end{subfigure}
	\caption{On the left, a balloon given by $C = \torg \ \mathtt{b} \ \tbl \ \mathtt{a} \ \tr \ \mathtt{b}\ \tw\ \mathtt{a}\ \tbl\ \mathtt{b}^{-1}\ \torg$ based at $a_0 = \torg$, and generators $\mathtt{a},\mathtt{b}\in S$. On the right, a portion of a configuration from $X_{\Fo}$ obtained by repeating the motif defined by the vertices of the balloon.}
		\label{fig:balloon}
\end{figure}
\vspace{-0.3cm}
\begin{definition}
	We say the set of forbidden patterns is \emph{complete} if there exists $\Col(A)\subseteq A$ and a map $f\colon\Col(A)\times (S\cup S^{-1})\to \Col(A)$ such that for all $a\in \Col(A)$, both $(a, f(a,s), s)$ and $(f(a,s^{-1}), a, s)$ are edges in $\Gamma_{\Fo}$ for $s\in S$.\footnote{This is also known as condition $(\star)$ in \cite{hellouin2020necessary}}
\end{definition}
	
Piantadosi showed in \cite{piantodosi2008free} that $X_{\Fo}$ is non-empty if and only if $\Fo$ is complete. Furthermore, $\Col(A)$ is computable from $A$ and $\Fo$, and every letter in a configuration $x\in X_{\Fo}$ is contained in $\Col(A)$.

\begin{lemma}
\label{lem:equiv}
	There exists a configuration $x\in X_{\Fo}$ containing $a_0$ infinitely many times if and only if there exists a simple balloon $B$ in $\Gamma_{\Fo}$ based at $a_0$, whose vertices are all in $\Col(A)$.
\end{lemma}

\begin{proof}
	Suppose we have a simple balloon $B = a_0 s_1 a_1 \ ... \ s_{n-1} a_{n-1} s_n a_0$ in $\Gamma_{\Fo}$ based at $a_0$ with $a_i\in\Col(A)$ and label $w = uvu^{-1}$ where $u = s_1 \ ... \ s_k$ with $k\leq\lceil\frac{n}{2}\rceil-1$. The condition over $k$ implies that $v\neq\varepsilon$. We define a configuration $x\in X_{\Fo}$ as follows: for every $t\in\N$, $x(\overline{w}^t) = a_0$ and $x(\overline{w^ts_1\ ...\ s_i}) = a_i$ with  $i\in\{1,\ ...,\ n-1\}$ (see Figure \ref{fig:balloon}). Because $B$ is a balloon, $x$ is well defined, as the balloon's definition guarantees 
	$$x(\overline{w^tuv}s_k^{-1}\ ...\ s_i^{-1}) = a_{n-i} = a_i = x(\overline{w^{t+1}s_1\ ...\ s_i}).$$
	Finally, because every letter belongs to $\Col(A)$, the rest of the configuration can be completed without forbidden patterns. Therefore, $x\in X_{\Fo}$.\\
	
	Conversely, suppose there exists $x\in X_{\Fo}$ where $a_0$ occurs infinitely often. Without loss of generality we can assume $x(1_{\F_n}) = a_0$. Recall that $x(\F_n)\subseteq\Col(A)$. Let us denote the set of elements $w\in\F_n$ where $x(w) = a_0$ by $\mathcal{O}$. Because $\mathcal{O}$ is infinite, there exists $s_0\in S\cup S^{-1}$ such that infinitely many words in $\mathcal{O}$ begin with $s_0$. Furthermore, there exists $s_1\in S\cup S^{-1}$ with $s_1 \neq s_0^{-1}$ such that infinitely many words in $\mathcal{O}$ begin with $s_0s_1$. By iterating this argument, we obtain a one-way infinite sequence $y\in (S\cup S^{-1})^\N$ such that $y(i) \neq y(i+1)^{-1}$ for all $i\in\N$. Let $\omega(i) = y(0)\ ...\ y(i-1)\in\F_n$. By definition, for every $i\in\N$ there are infinitely many words in $\mathcal{O}$ that begin with $\omega(i)$. We will say $w\in\mathcal{O}$ is \emph{rooted} at $i\in\N$ if $w = \omega(i)v$ for some $v\in (S\cup S^{-1})^*$, and such that the concatenation is reduced. Because there are infinitely many words rooted along some point of $y$, and $\Col(A)$ is finite, there exist $j_1 < i < j_2$ and $w\in\mathcal{O}$ such that $x(\omega(j_1)) = x(\omega(j_2))$ and $w$ is rooted at $i$. Using this fact, we will create a balloon depending on two cases.
	\begin{enumerate}
		\item If $y(j_1 - 1)\neq y(j_2 -1)$, and calling $a_j = x(\omega(j))$, define the balloon that represents going from $x(1_G)$ to $x(\omega(j_2))$ by the path $\omega(j_2)$ and then return via the path $\omega(j_1)$ (see Figure~\ref{fig:caso1}). Formally, $$B = a_0\ y(0)\ a_1\ ... \ y(j_1 -1)\ a_{j_1}\ ...\ y(j_2 -1)\ a_{j_2}\ y(j_1-1)^{-1}\ a_{j_1-1}\ ...\ y(0)^{-1} a_0,$$ which is labeled by $\omega(j_2)\omega(j_1)^{-1}$, a reduced word.
		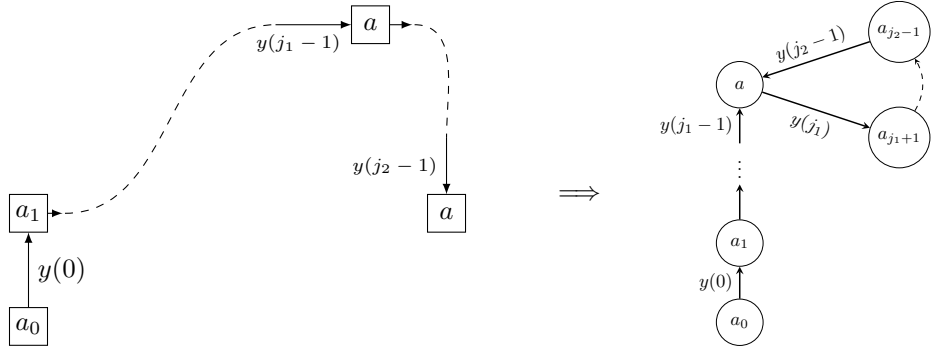
\begin{figure}[h]
			\centering
			\begin{subfigure}{0.6\textwidth}
				\begin{tikzpicture}

\draw (1,1) -- (1,1.5) -- (1.5,1.5) -- (1.5,1) -- cycle;
\draw (1.25,1.25) node [anchor=center] {$a_0$};

\draw [-latex] (1.25, 1.5) -- (1.25,2.5) node[midway,right] {$y(0)$};

\draw (1,2.5) -- (1,3) -- (1.5,3) -- (1.5,2.5) -- cycle;
\draw (1.25,2.75) node [anchor=center] {$a_1$};

\draw[-latex] (1.5,2.75) -- (1.7,2.75);
\draw[dashed] (1.73, 2.75) to[out=0,in=180] (4.5,5.25);

\draw (5.5,5) -- (5.5,5.5) -- (6,5.5) -- (6,5) -- cycle;
\draw (5.75,5.25) node [anchor=center] {$a$};

\draw[-latex] (4.5,5.25) -- (5.5,5.25) ;
\draw (4.8,5) node [anchor=center] {\scriptsize{$y(j_1-1)$}};

\draw (6.5,2.5) -- (6.5,3) -- (7,3) -- (7,2.5) -- cycle;
\draw (6.75,2.75) node [anchor=center] {$a$};

\draw[-latex] (6.75,3.75) -- (6.75,3) node[midway,left] {\scriptsize{$y(j_2-1)$}};
\draw[dashed] (6.33,5.25) to[out=0,in=90] (6.75,3.75);
\draw[-latex] (6,5.25) -- (6.3,5.25);

\draw (8.5,3) node [anchor=center]  {$\implies$};

\end{tikzpicture}
			\end{subfigure}
			\begin{subfigure}{0.3\textwidth}
				\begin{tikzpicture}[node distance = 1cm, on grid, auto]
		
	\node (q0) [state] at (2,0) {$a_0$};
	\node (q1) [state] at (2,1.5) {$a_1$};
	\node (qdots1) at (2,3) {$\vdots$};
	\node (q2) [state] at (2,4.5) {$a$};
	\node (q3) [state] at (5,3.5) {$a_{j_1+1}$};
	\node (q4) [state] at (5,5.5) {$a_{j_2-1}$};
	\node (qdots2) at (7,4.5) {$ $};

	\path [-stealth, thick]
	(q0) edge node {$y(0)$}   (q1)
	(q1) edge (qdots1)
	(qdots1) edge node {$y(j_1-1)$} (q2)
	(q2) edge node[below, sloped] {$y(j_1)$} (q3)
	(q4) edge node[above, sloped] {$y(j_2 -1)$}(q2);

	\path [-stealth,dashed]
	(q3) edge [bend right] (q4);

\end{tikzpicture}
			\end{subfigure}
			\caption{On the left, the path defined by $y$ in the configuration. This is an example of the first case, where $y(j_1-1)\neq y(j_2-1)$ and the repeated letter is $a = x(\omega(j_1)) = x(\omega(j_2))$. On the right, the corresponding balloon within the tileset graph $\Gamma_{\Fo}$.}
			\label{fig:caso1}
		\end{figure}	
		
		\item If $y(j_1 - 1) = y(j_2 -1)$, then $y(j_1)\neq y(j_2 -1)^{-1}$. Let $v\in (S\cup S^{-1})^k$ such that $w = \omega(i)v$. Once again, calling $a_j = x(\omega(j))$ and $b_j = x(\overline{wv_1 ... v_j})$, we define the balloon
		$$B = a_0\ v_k^{-1}\ b_{k-1}\ v_{k-1}^{-1}\ ...\ v_1^{-1}\ a_i \ y(i)\ ...\ y(j_2 -1)\ a_{j_2}\ y(j_1)\ ...\ y(i-1)\ a_i\ v_1\ b_1\ ...\ v_k\ a_0,$$
		which is labelled by $v y(i)\ ...\ y(j_2 -1)y(j_1)\ ...\ y(i-1) v^{-1}$, a reduced word (see Figure~\ref{fig:caso2}).
		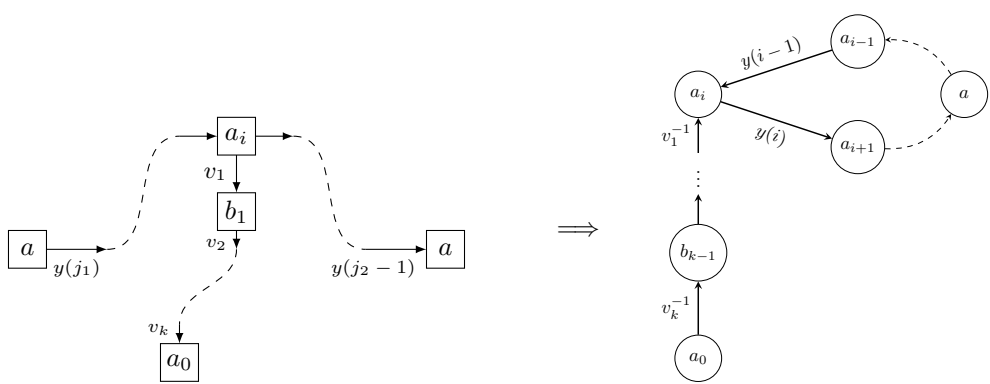
\begin{figure}[h]
			\centering
			\begin{subfigure}{0.6\textwidth}
				\begin{tikzpicture}

\draw (3,1) -- (3,1.5) -- (3.5,1.5) -- (3.5,1) -- cycle;
\draw (3.25,1.25) node [anchor=center] {$a_0$};
\draw [-latex] (3.25, 1.75) -- (3.25,1.5) node [pos=0.25,left] {\scriptsize{$v_k$}};

\draw [-latex] (1.5, 2.75) -- (2.3,2.75) node[midway,below] {\scriptsize{$y(j_1)$}};

\draw (1,2.5) -- (1,3) -- (1.5,3) -- (1.5,2.5) -- cycle;
\draw (1.25,2.75) node [anchor=center] {$a$};

\draw[dashed] (2.3, 2.75) to[out=0,in=180] (3.25,4.25);

\draw (3.75,4) -- (3.75,4.5) -- (4.25,4.5) -- (4.25,4) -- cycle;
\draw (4,4.25) node [anchor=center] {$a_i$};

\draw[-latex] (3.25,4.25) -- (3.75,4.25);
\draw[-latex] (4.25,4.25) -- (4.75,4.25);

\draw (6.5,2.5) -- (6.5,3) -- (7,3) -- (7,2.5) -- cycle;
\draw (6.75,2.75) node [anchor=center] {$a$};
\draw (5.8,2.5) node [anchor=center]  {\scriptsize{$y(j_2-1)$}};
\draw [-latex] (5.7, 2.75) -- (6.5,2.75);

\draw[dashed] (4.75,4.25) to[out=0,in=180] (5.7,2.75);

\draw[-latex] (4,4) -- (4,3.5) node [midway, left] {\small{$v_1$}};

\draw (3.75,3) -- (3.75,3.5) -- (4.25,3.5) -- (4.25,3) -- cycle;
\draw (4,3.25) node [anchor=center] {$b_1$};
\draw [-latex] (4, 3) -- (4,2.75) node [pos=0.75,left] {\scriptsize{$v_2$}};
\draw [dashed] (4,2.75) to[out=-90, in=90] (3.25,1.75);

\draw (8.5,3) node [anchor=center]  {$\implies$};

\end{tikzpicture}
			\end{subfigure}
			\begin{subfigure}{0.3\textwidth}
				\begin{tikzpicture}[node distance = 1cm, on grid, auto]
		
	\node (q0) [state] at (2,0) {$a_0$};
	\node (q1) [state] at (2,2) {$b_{k-1}$};
	\node (qdots1) at (2,3.5) {$\vdots$};
	\node (q2) [state] at (2,5) {$a_i$};
	\node (q3) [state] at (5,4) {$a_{i+1}$};
	\node (q4) [state] at (5,6) {$a_{i-1}$};
	\node (q5) [state] at (7,5) {$a$};

	\path [-stealth, thick]
	(q0) edge node {$v_k^{-1}$}   (q1)
	(q1) edge (qdots1)
	(qdots1) edge node {$v_1^{-1}$} (q2)
	(q2) edge node[below, sloped] {$y(i)$} (q3)
	(q4) edge node[above, sloped] {$y(i-1)$}(q2);

	\path [-stealth,dashed]
	(q3) edge [bend right] (q5)
	(q5) edge [bend right] (q4);

\end{tikzpicture}
			\end{subfigure}
			\caption{On the left, the path defined by $y$ in the configuration as well as the path leading from the root $x(\omega(i))$ to $w$. This is an example of the second case, where $y(j_1-1)=y(j_2-1)$ and the repeated letter is $a = x(\omega(j_1)) = x(\omega(j_2))$. On the right, the corresponding balloon within the tileset graph $\Gamma_{\Fo}$.}
			\label{fig:caso2}
		\end{figure}
	\end{enumerate}
	\vspace{-0.2cm}
	Finally, if $B$ contains a repeated letter/generator pair, we can simply cut the portion between them while preserving all other balloon conditions. This guarantees that $B$ will be a simple balloon in $\Gamma_{\Fo}$ based at $a_0$, with all its vertices in $\Col(A)$.

\end{proof}

\begin{proof}[Proof of Theorem \ref{thm:free}]
	Given a finite set of nearest neighbor forbidden patterns $\Fo$ and a letter $a_0$, by Lemma \ref{lem:equiv}, it suffices to search for simple balloons in $\Gamma_{\Fo}$ whose vertices are contained in $\Col(A)$. A simple balloon passes through each vertex-label pair at most once; and there are a finite number of such paths starting and ending in $a_0$, so we can check whether they satisfy the simple balloon conditions. Therefore, we can effectively decide whether the conditions of Lemma \ref{lem:equiv} are met, making the recurrence problem decidable.
\end{proof}

\subsection{Consequences and Conjectures}

As previously stated, we are interested in understanding the class of groups that have decidable seeded domino problem, and the class of groups that have decidable recurring domino problem.

\begin{theorem}
	Let $G$ be a virtually free group. Then, both $\SDP(G)$ and $\RDP(G)$ are decidable.
\end{theorem}

\begin{proof}
	By Theorem \ref{thm:free} we know the recurring domino problem is decidable on free groups. Adding Lemma \ref{lem:fi}, we have that it is decidable for virtually free groups. For the seeded version, as we mentioned earlier, we have that the problem is expressible in MSO logic and is therefore decidable for virtually free groups.
\end{proof}

Are these the only groups where each individual problem is decidable? The combination of Conjecture \ref{conj:domino} and Lemmas \ref{lem:dp_to_sdp} and \ref{lem:dp_to_rdp} suggest so.

\begin{corollary}
	If the Domino Conjecture is true the following are equivalent:
	\begin{itemize}
		\item $G$ is virtually free,
		\item $\DP(G)$ is decidable,
		\item $\SDP(G)$ is decidable,
		\item $\RDP(G)$ is decidable.
	\end{itemize}
\end{corollary}

Nevertheless, virtually free groups being the only groups where the seeded domino problem or the recurring domino problem is decidable does not directly imply Conjecture \ref{conj:domino}. We can nonetheless state a conjecture for the seeded domino problem due to the fact it can be expressed in MSO logic.

\begin{conjecture}
	Let $G$ be a finitely generated group. Then, $\SDP(G)$ is decidable if and only if $G$ is virtually free.
\end{conjecture}

\section{The $k$-SAT Problem on Groups}

As mentioned in the introduction, we define a generalized version of the $k$-$\SAT$ problem for finitely generated groups. This version is slightly different from the one introduced by Freedman~\cite{freedman1999sat} in order to correctly capture the structure of finitely generated groups.\\

Let $G$ be a finitely generated group and $H\leq G$ a finitely generated subgroup. As variables for our formulas we use elements of $G$. For $g\in G$, we denote its negation by $\neg g$ and we use the ambiguous notation $g'$ to refer to either $g$ or $\neg g$ depending on the formula. We denote the set of formulas over $G$ containing $k$ literals as $N_k$, that is, $\phi\in N_k$ if
$$ \phi = \bigwedge_{i=1}^{m}\left((g_{i1})' \vee \ ... \ \vee (g_{ik})' \right).$$

Next, we define the set of formulas $HN_k$ as all the formulas of the form:
$$\bigwedge_{h\in H}\bigwedge_{i=1}^{m}\left((hg_{i1})' \vee \ ... \ \vee (hg_{ik})' \right)$$

We use $\phi(h)$ to denote the formula $\phi$ with each literal left-multiplied by $h$.

\begin{definition}
	We say a formula $\phi\in HN_k$ is satisfiable, if there exists an assignment of truth values $\alpha:G\to\{0,1\}$ such that:
	$$\bigwedge_{h\in H}\bigwedge_{i=1}^{m}\left(\alpha(hg_{i1})' \vee \ ... \ \vee \alpha(hg_{ik})' \right) = 1.$$
\end{definition}

Let $S$ be a finite generating set for $G$. To arrive at a valid decision problem, we will specify a function by a set of words over $S\cup S^{-1}$ that will evaluate to the literals of the function, and a list of words, also over $S\cup S^{-1}$, that will specify a generating set for a subgroup. Formally, an \emph{input formula} is a formula of the form

$$\phi = \bigwedge_{i=1}^m(v_{i1}' \vee \ ... \ \vee v_{ik}'),$$

where $v_{ij}\in (S\cup S^{-1})^*$ for all $i\in \{1,...,m\}$ and $j\in\{1,...,k\}$, such that its evaluated version 
$$\bar{\phi} = \bigwedge_{i=1}^m(\bar{v}_{i1}' \vee \ ... \ \vee \bar{v}_{ik}')$$
belongs to $N_k$.

\begin{definition}
	Let $G$ be a finitely generated group, $S$ a finite generating set and $k>1$. The \emph{$k$-$\SAT$ problem over $G$} is the decision problem that given an input formula $\phi$ and $\{w_i\}_{i=1}^n$ determines if the formula $\bigwedge_{h\in H}\bar{\phi}(h)$ is satisfiable, where $H = \langle w_1, ..., w_n\rangle$.
\end{definition}

Notice that the decidability of this problem does not depend on the chosen generating set, as we can re-write any input into any other generating set. We therefore denote this problem by $k$-$\SAT(G)$.

The first observation to make is that this problem depends on the computational structure of the group.

\begin{lemma}
	The subgroup membership problem of $G$ many-one reduces to $\textnormal{co}2\text{-}\SAT(G)$.
\end{lemma}

The subgroup membership problem of a f.g. group $G$ is the decision problem that takes as input words $u, \{w_i\}_{i=1}^n$ over $S\cup S^{-1}$ and determines if $\bar{u}\in\langle w_1, ..., w_n\rangle$.

\begin{proof}
	Let $u, \{w_i\}_{i=1}^n\in (S\cup S^{-1})^*$ be an instance of the subgroup membership problem. We define the formula 
	$$\psi = (\neg\varepsilon \vee s)\wedge (u \vee s)\wedge (\neg\varepsilon \vee \neg s)\wedge (u \vee \neg s),$$
	
	 for some fixed $s\in S$, which along with the words $w_i$ is an input to $2$-$\SAT(G)$. Notice that $\psi$ is equivalent to the formula $\neg\varepsilon \wedge u$. Let us denote $H = \langle w_1, ..., w_n\rangle$ and $\Psi = \bigwedge_{h\in H}\bar{\psi}(h)$.
	
	Suppose $\bar{u}\in H$. Then, we have that $\bar{\psi}(1_G) \wedge \bar{\psi}(\bar{u}) = (\neg 1_G \wedge \bar{u}) \wedge (\neg\bar{u}\wedge \bar{u}^2)$ is never satisfiable, and thus $\Psi$ is not satisfiable. On the other hand, if $\bar{u}\not\in H$, we can define the assignment $\alpha\colon G\to\{0,1\}$ by $\alpha(h) = 0$ and $\alpha(hu) = 1$ for all $h\in H$, and $\alpha(g) = 0$ for all other $g\in G\setminus H$. This way $\neg\alpha(h) \wedge \alpha(hu) = 1$ for all $h$, and therefore $\Psi$ is satisfied.
\end{proof}

Examples of groups with undecidable subgroup membership problems are $\F_n\times\F_n$ \cite{mihailova1968occurrence}, some hyperbolic groups \cite{rips1982small}, as well as groups with undecidable word problem.

\begin{lemma}
	\label{lem:sat_to_dom}
	Let $G$ be a finitely generated group with decidable subgroup membership problem.
	Then, for every $k\geq 2$ we have that $k$-$\SAT(G)\leq_m\DP(G)$.
\end{lemma}

\begin{proof}
	Let $S$ be a finite generating set for $G$, $\phi$ an input formula and $\{w_i\}_i$ words over $S\cup S^{-1}$ that form an instance of $k$-SAT$(G)$ such that 	
	$$\phi = \bigwedge_{i=1}^m(v_{i1}' \vee \ ... \ \vee \ v_{ik}'),$$
	with $v_{ij}\in (S\cup S^{-1})^*$. Let us once again denote $H = \langle w_1, ..., w_n\rangle$. 
	Our alphabet, $A$, consists of 0-1 matrices of size $m\times k$ that satisfy $\phi$, that is, all matrices $M\in\{0,1\}^{m\times k}$ such that
	$$\bigwedge_{i=1}^{m}\left((M_{i1})' \vee \ ... \ \vee (M_{ik})' \right) = 1.$$
	To obtain this alphabet we must solve the standard $k$-SAT problem, which is computable.
	
	For convenience, let us denote the finite subset of words involved in the formula by $L = \{v_{ij}\mid 1\leq i\leq m, \ 1\leq j \leq k\}$. In addition, we define the set $H_{L}$ as the set of all $h_{abcd}\in H\cap LL^{-1}$, where
	$$h_{abcd} = \begin{cases} v_{ab}v_{cd}^{-1} \ \text{ if } \ v_{ab}v_{cd}^{-1}\in H, \\ 1_H \ \ \text{ otherwise. }\end{cases}$$
	
	Notice that $|H_L|\leq |L|^2 = m^2k^2$, and that this set is computable as $G$ has decidable subgroup membership problem. Let us proceed by specifying a set of nearest neighbor forbidden rules, $\Fo$, with respect to the generating set $S\cup H_L$. Given a configuration $x\in X_\Fo$ the idea is that, for $h\in H$, the matrix $x(h)$ will stock the values assigned to the elements of $hL$. For each $h_{abcd}\in H_L$, we forbid patterns $q$ of support $\{1_H,h_{abcd}\}$, such that if $q(1_H) = M$ and $q(h_{abcd}) = \hat{M}$,
		$$M_{ab} \neq \hat{M}_{cd}.$$
			
	Suppose $X_\Fo$ contains a configuration $x$. The assignment of truth values $\alpha:G\to\{0,1\}$ is defined by
	$$\alpha(g) = \begin{cases} 0 \ \text{ if } \ g\not\in H\cdot L, \\ 
		x(h)_{ab} \ \text{ if } \ g = h\bar{v}_{ab}\end{cases}.$$
	
	It follows that $\alpha$ is well defined; if $g = h_1\bar{v}_{ab} = h_2\bar{v}_{cd}$, then $h_2 = h_1 h_{abcd}$, and by the forbidden patterns we know $(x_{h_1})_{ab} = (x_{h_2})_{cd}$. In addition, because $x\in A^G$,  for all $h\in H$, 
	$$\bigwedge_{i=1}^{m}\left(\alpha(h\bar{v}_{i1})' \vee \ ... \ \vee \alpha(h\bar{v}_{ik})' \right) = \bigwedge_{i=1}^{m}\left(x(h)_{i1}' \vee \ ... \ \vee x(h)_{ik}' \right) = 1.$$
	This means that the assignation $\alpha$ satisfies $\bigwedge_{h\in H}\bar{\phi}(h)$.\\
	
	Finally, suppose we have an assignation of truth values $\beta:G\to\{0,1\}$ that satisfies $\bigwedge_{h\in H}\bar{\phi}(h)$. Given a set of right coset representatives $R$ containing $1_G$, we define $z\in\{0,1\}^{m\times k}$ by $z(hr)_{ab} = \beta(hg_{ab})$, for all $h\in H$ and $r\in R$. Because $\beta$ satisfies $\bigwedge_{h\in H}\bar{\phi}(h)$, for all $h\in H$
	$$\bigwedge_{i=1}^{m}\left(z(h)_{i1}' \vee \ ... \ \vee z(h)_{ik}' \right) =  \bigwedge_{i=1}^{m}\left(\beta(h\bar{v}_{i1})' \vee \ ... \ \vee \beta(h\bar{v}_{ik})' \right) = 1.$$
	Therefore, $z\in A^G$.  For $h_{abcd}\in H_L$, $h_1\in H$ and $h_2 = h_1h_{abcd}$ we have that 
		$$z(h_1)_{ab} = \beta(h_1\bar{v}_{ab}) = \beta(h_1h_{abcd}\bar{v}_{cd}) = \beta(h_2\bar{v}_{cd}) =  z(h_2)_{cd}$$
	Therefore $z$ satisfies the local rules and is thus in $X_\Fo$. This concludes our reduction.
\end{proof}

Virtually free groups not only have decidable domino problem, as previously mentioned, but also have decidable subgroup membership problem (see \cite{lohrey2023subgroup}).

\begin{corollary}
	For $G$ a virtually free group, $k$-$\SAT(G)$ is decidable for all $k>1$.
\end{corollary}

To determine when the converse reduction is true, we introduce a new class of groups that has the required properties.

\begin{definition}
	Let $G$ be a finitely generated group. We say $G$ is \emph{scalable} if there exists a proper finite index subgroup $H\lneq G$ that is isomorphic to $G$.
\end{definition}

Examples of such groups are finitely generated abelian groups, the Heisenberg group, solvable Baumslag-Solitar groups $BS(1,n)$, Lamplighter groups $F\wr \Z$ with $F$ a finite abelian group, the affine group $\Z^d\rtimes GL(d,\Z)$ for $d\geq 2$ \cite{nekrashevych2011scale}, among others. Examples of non-scalable groups are finitely generated free groups.

\begin{theorem}
	For $G$ a scalable group, $\DP(G)\leq_m 3$-$\SAT(G)$.
\end{theorem}

\begin{proof}
	Let $A$ be a finite alphabet of size $n$, and $\Fo$ a finite set of nearest neighbor forbidden patterns for $G$ with generating set $S$. As $G$ is scalable, there exists $H$ a proper subgroup of finite index as well as an isomorphism $F\colon G\to H$. Let $f\colon S\to S^{*}$ the function that is extended to the isomorphism $F$, that is, $\{f(s)\}_{s\in S}$ is a generating set for $H$. Fix $R\subseteq (S\cup S^{-1})^{*}$ a set of words representing a finite set of right coset representatives for $H$ that includes $1_G$. Notice that for every $m\in\N$ the subgroup $H_m = F^{m}(G)$ is isomorphic to $G$ with $[G:H_m] = [G:H]^m\geq m$. In addition, a simple computation shows that $R_m = F^{m -1}(R) \ ... \ F(R)R$ defines a set of right coset representatives for $H_m$. \\
	
	The idea of the reduction is to represent each letter of the alphabet by a unique code on the left coset representatives and then create a formula that assigns a letter to each element of $H_m$. The index of the subgroup, $m$, will be chosen so there is enough room to code the alphabet and write our formula in the required form.
	
	First off, take as a preliminary estimate $m \geq \lceil\log_2(n)\rceil$, and denote $f_m = f^m$. Let $\{\phi_a\}_{a\in A}$ be the set of formulas that code the elements of the alphabet $A$ using the words in $R_m$ as variables. This way $\phi_a(h) \equiv 1$ means we place the letter $a$ at $g = (F^m)^{-1}(h)$, and the variables are contained in $hR_m$. Our formula is given by,
	
	$$\varphi = \left(\bigvee_{a\in A} \phi_a(1_G)\right) \wedge \left(\bigwedge_{(a,b,s)\in\Fo}\neg\phi_a(1_G)\vee \neg\phi_{b}(f_m(s))\right),$$
	
	which represents the fact that we place one letter at the given point ($1_G$ in this case) and that there are no forbidden patterns in its neighborhood. If modified to be in CNF form, $\varphi$ is a conjunction of $|\Fo| + \lceil\log_2(n)\rceil^n$ clauses of $\leq n$ literals (the clauses coding the forbidden patterns contain $2\lceil\log_2(n)\rceil$ literals). By adding $(|\Fo| + \lceil\log_2(n)\rceil^n)n$ dummy variables we can transform $\varphi$ into an equivalent formula $\varphi'$ whose clauses contain exactly 3 literals.
	
	Therefore, take $m \geq (|\Fo| + \lceil\log_2(n)\rceil^n)n + \lceil\log_2(n)\rceil$, which gives us enough space in the set of left coset representatives to code the elements of the alphabet and the dummy variables. Furthermore, $\varphi'$ is computable from $A$ and $\Fo$, and $\Phi' = \bigwedge_{h\in H}\bar{\varphi}'(h)\in H_mN_3$. 
	
	Let us prove the reduction. If there exists $x\in X_{\Fo}\subseteq A^G$, we create an assignment such that for all $g\in G$, the variables in $F^m(g)R_m$ are given values so as to satisfy the code for $\phi_{x(g)}(F^m(g))\equiv 1$. Because $x$ contains no patterns from $\Fo$, $\bigwedge_{h\in H}\bar{\varphi}(h)$ will be satisfied. We finish by filling out the rest of the variables so that $\Phi'\equiv 1$.
	
	Now, if $\Phi'$ is satisfied so is $\bigwedge_{h\in H}\bar{\varphi}(h)$. Let $y\in A^G$ be the configuration defined by $y(g) = a$ if $\phi_a(F^m(g))\equiv 1$. Because the codes used make sure that the values in $F^m(g)R_m$ code a unique letter, for each $g\in G$ a unique $\phi_a(F^m(g))$ is satisfied. Thus $y$ is well defined. Finally, $y\in X_{\Fo}$ because if there was $g\in G$ such that $y(g) = a$ and $y(gs) = b$ with $(a,b,s)\in\Fo$ we would have that $\phi_a(F^m(g)) \wedge \phi_b(F^m(g)f_m(s))$ is true. This shows $\DP(G)\leq_m 3$-$\SAT(G)$. 
\end{proof}

\begin{corollary}
	$3$-$\SAT(G)$ is undecidable for finitely generated abelian groups, solvable Baumslag-Solitar groups, the Heisenberg group and affine groups.
\end{corollary}


\bibliography{biblio}
\end{document}